\documentclass{article}
\usepackage[utf8]{inputenc}
\usepackage{amsmath}
\usepackage{amssymb}
\usepackage{amsthm}
\usepackage{verbatim}
\usepackage{enumerate}
\usepackage{amsfonts}
\usepackage{amsmath}
\usepackage{hyperref}
\usepackage[noadjust]{cite}

\DeclareMathOperator{\ex}{ex}
\newtheorem{theorem}{Theorem}

\newtheorem{lemma}{Lemma}

\title{Edge colorings of graphs without monochromatic stars}

\author{
  Lucas Colucci$^{1,2}$\\
  \texttt{lcolucci@ime.usp.br}
  \and
  Ervin Gy\H ori$^{1,2}$\\
  \texttt{gyori.ervin@renyi.mta.hu}
  \and 
  Abhishek Methuku$^{3}$\\
  \texttt{abhishekmethuku@gmail.com}
}

\date{  $^1$Alfréd Rényi Institute of Mathematics, Hungarian Academy of Sciences, Re\'altanoda~u.~13--15, 1053 Budapest, Hungary\\%
        $^2$Central European University, Department of Mathematics and its Applications, N\'ador~u.~9, 1051 Budapest, Hungary\\
        $^3$Department of Mathematics, École polytechnique fédérale de Lausanne, EPFL FSB SMA, Station 8, 1015 Lausanne, Switzerland\\[2ex]%
        \today}

\begin{document}

\maketitle

\begin{abstract}
   In this note, we improve on results of Hoppen, Kohayakawa and Lefmann about the maximum number of edge colorings without monochromatic copies of a star of a fixed size that a graph on $n$ vertices may admit. Our results rely on an improved application of an entropy inequality of Shearer.
\end{abstract}

\section{Introduction}

Let $r$ be a positive integer and $G$ and $H$ be (simple) graphs. We define $c_{r,H}(G)$ as the number of $r$-edge-colorings of $G$ (i.e., functions $c:E(G) \rightarrow [r]=\{1,\dots,r\}$) without a monochromatic copy of $H$ as a subgraph. For instance, when $H$ is the path on $3$ vertices (we denote it by $P_3$), $c_{r,H}(G)$ is simply the number of proper $r$-edge-colorings of $G$. Furthermore, let $c_{r,H}(n)$ be the maximum value of $c_{r,H}(G)$ as $G$ runs through all graphs on $n$ vertices. A graph $G$ is called \emph{$(r,H)$-extremal} if $c_{r,H}(G)=c_{r,H}(|V(G)|)$.

\medskip

For every $r$, $n$ and $H$, we have the following general bounds: 
\begin{equation}\label{trivial}
    r^{\ex(n,H)}\leq c_{r,H}(n)\leq r^{r\cdot \ex(n,H)},
\end{equation}

where $\ex(n,H) = \max\{e: \text{ there is $G$ with $n$ vertices, $e$ edges and $H \nsubseteq G$}\}$ is the classical extremal (or Turán) number of $H$.

The lower bound is obtained by taking $G$ as an $H$-free graph on $n$ vertices and $\ex(n,H)$ edges (i.e., an $H$-free extremal graph); the upper bound follows from the fact that in any $r$-coloring of a graph on $n$ vertices and at least $r\cdot \ex(n,H)+1$ edges there is a monochromatic subgraph on at least $\ex(n,H)+1$ edges, by the Pigeonhole Principle, and hence a monochromatic $H$.

\medskip

This problem traces back to a question of Erd\H os and Rothschild (\cite{erdos1992some}) that corresponds to $r=2$ and $H=K_3$ in the setup above. More precisely, they conjectured that $c_{2,K_3}(n)$ matches the lower bound in (\ref{trivial}) for all $n$ large enough, which was proved by Yuster:

\begin{theorem}\cite{yuster1996number}
$c_{2,K_3}(n)=2^{\lfloor{n^2/4}\rfloor}$ for all $n\geq 6$.
\end{theorem}

He conjectured further that the same result holds for $H=K_t$ and proved an asymptotic version of the conjecture, which was settled later by Alon, Balogh, Keevash and Sudakov for 2 and 3 colors:

\begin{theorem}\cite{alon2004number}
For every fixed $t$, there is $n_0$ such that $c_{2,K_t}(n)=2^{\ex(n,K_t)}$ and $c_{3,K_t}(n)=3^{\ex(n,K_t)}$ hold for $n>n_0$.
\end{theorem} 

Their proof uses Szemerédi's Regularity Lemma, and hence the value of $n_0$ it gives is a tower type with height exponential in $k$. More recently, Hàn and Jiménez \cite{han2018improved} improved $n_0$ to an exponential function of $k$, namely $\exp(Ck^4)$, getting much closer to the lower bound of $\exp(Ck)$ mentioned in \cite{alon2004number}.

\medskip

They also dealt with the case $r>3$, showing that the lower bound in (\ref{trivial}) is not the correct value of $c_{r,K_t}(n)$ in this case, i.e., the $K_t$-free Turán graphs are not the $(r,K_t)$-extremal graphs in this case. We refer to their paper \cite{alon2004number} for the detailed results.

\medskip

Pikhurko and Yilma \cite{pikhurko2012maximum} determined the $(r,K_t)$-extremal graphs $r = 4$, $t = 3, 4$ and $n$ sufficiently large. They are (not $K_t$-free) Turán graphs. Together with Staden \cite{pikhurko2017erdHos}, they generalized it to the following: we want to color the edges of a graph on $n$ vertices using $s$ colors in a way that, for every $1 \leq i \leq s$, there is no monochromatic $K_{t_i}$ of color $i$. They proved that for any choice of $n$, $s$ and $t_i$, there is a complete multipartite graph that attains the maximum number of colorings. A similar result is proved in \cite{benevides2017edge}, where a fixed pattern of colors (not necessarily monochromatic) in a clique is forbidden.

\medskip

Considering the disjoint union of two $(r,H)$-extremal graphs on $n$ and $m$ vertices, it is easy to see, assuming $H$ is a connected graph, that $c_{r,H}(n+m)\geq c_{r,H}(n) \cdot c_{r,H}(m)$ holds for all positive integers $m$ and $n$ (i.e., the function $c_{r,H}(n)$ is supermultiplicative).  A lemma of Fekete (\cite{fekete1923verteilung}) implies, then, that the limit $b_{r,H}=\lim_{n\rightarrow\infty}c_{r,H}(n)^{1/n} \in \mathbb{R}\cup\{\infty\}$ exists.

\medskip

Hoppen, Kohayakawa and Lefmann addressed the problem for some graphs $H$ with linear Turán number (i.e., $\ex(n,H)=O(n)$). By (\ref{trivial}), these are exactly the graphs for which $b_{r,H}$ is finite. They settled the question when $H$ is a matching of fixed size (\cite{hoppen2012edge}), and studied it for other classes of bipartite graphs, including paths and stars (\cite{hoppen2014edge}). Surprisingly, only very few exact values of $b_{r,H}$ are known in these cases. In this note, we will improve some of the current upper bounds when the forbidden graph is a star. We now state the best known upper and lower bounds followed by our corresponding improvements on the upper bounds in each case.

\medskip

First, we consider small forbidden stars ($S_3$ and $S_4$) and 2-colorings, where $S_t$ is the star on $t$ edges (and $t+1$ vertices). For $S_3$, Hoppen, Kohayakawa and Lefmann had the following bounds:

\begin{theorem}\cite{hoppen2014edge}    
$b_{2,S_3}\leq \sqrt{6} \approx 2.45$. On the other hand, the graph consisting of $n/6$ disjoint copies of the complete bipartite graph $K_{3,3}$ gives $b_{2,S_3} \geq \sqrt[6]{102} \approx 2.16$.
\end{theorem}

We improve the upper bound above to:

\begin{theorem}\label{S3}
There is a constant $c$ such that $c_{2,S_3}(n)\leq c\cdot18^{3n/10}$. In particular, $b_{2,S_3}\leq 18^{3/10}\approx 2.38.$
\end{theorem}

Their result for $S_4$ is:

\begin{theorem}\cite{hoppen2014edge}
$b_{2,S_4}\leq \sqrt{20}\approx 4.47$. On the other hand, the graph consisting of the union of $n/10$ disjoint bipartite graphs $K_{5,5}$ gives $b_{2,S_4}\geq 3.61$.
\end{theorem}

Our improved upper bound in this case is:

\begin{theorem}\label{S4}
$b_{2,S_4}\leq 200^{5/18}\approx 4.36$.
\end{theorem}

Next, we consider 2-colorings that forbid monochromatic big stars. Hoppen, Kohayakawa and Lefmann, in the same paper, proved the following:

\begin{theorem}\cite{hoppen2014edge}
For every $t$, $b_{2,S_t}\leq \binom{2t-2}{t-1}^{1/2}$. Furthermore, a certain complete bipartite graph gives $b_{2,S_t}\geq 2^{-(\sqrt{2}/2+o(1))\sqrt{t\log(t)}}\cdot\left(\binom{2t-2}{t-1}\right)^{1/2}$. 
\end{theorem}

We improve the upper bound for large $t$ as follows:

\begin{theorem}\label{b2t}
For large values of $t$, we have:

\begin{center}
$b_{2,S_t}\leq \left(\frac{\sqrt{2}}{2}\cdot\binom{2t-2}{t-1}\right)^\frac{2t-3}{4t-7}$.
\end{center}

\end{theorem}

Finally, we fix the forbidden star to be $S_3$ and consider $r$-colorings. The bounds in Hoppen, Kohayakawa and Lefmann's paper are:

\begin{theorem}\cite{hoppen2014edge}
For every $r$, $b_{r,S_3}\leq\left(\frac{(2r)!}{2^r}\right)^{1/2}$. On the other hand, some complete bipartite graph shows that $b_{r,S_3}\geq r^{-(3\sqrt{\log(3)}/4+o(1))r}\cdot\left(\frac{(2r)!}{2^r}\right)^{1/2}$.
\end{theorem}

The new upper bound for this quantity that we prove here is:

\begin{theorem}\label{br3}
If $r$ is a sufficiently large integer, then 
$$b_{r,S_3}\leq \left(\frac{r(2r-1)!^2}{2^{2r-2}}\right)^{\frac{2r-1}{8r-6}}\sim \frac{\sqrt[8]{2}}{\sqrt[4]{e}}\cdot\left(\frac{(2r)!}{2^r}\right)^{1/2}\approx 0.85\cdot\left(\frac{(2r)!}{2^r}\right)^{1/2}.$$
\end{theorem}

\section{Notation and preliminary lemma}

Given a graph $G$, we call an edge $e=uv \in E(G)$ an $ab$-edge ($a \leq b$) if $\{d(u),d(v)\}=\{a,b\}$. Furthermore, we denote by $m_{ab}$ the number of $ab$-edges (sometimes we will write $m_a$ instead of $m_{aa}$ for short) and by $v_a$ the number of vertices of degree $a$ in $G$.

\medskip

We now state and prove a simple lemma that will be used throughout the proofs of this paper.

\begin{lemma}\label{degrees}
For every $r\geq 2$, $t\geq 3$ and $n$, there is an $(r,S_t)$-extremal graph $G$ on $n$ vertices and a constant $c(r,t)$ (which is at most $rt+1$) with the following properties: $\Delta(G)\leq r(t-1)-1$, and $d(v) \geq \left\lceil \frac{r}{2} \right\rceil\cdot(t-1)$ holds for all but at most $c(r,t)$ vertices $v \in V(G)$. 
\end{lemma}

\begin{proof} 

Let $G$ be a graph on $n$ vertices. If $G$ has a vertex of degree at least $r(t-1)+1$, all of its $r$-edge colorings contain a monochromatic $S_t$, by Pigeonhole Principle, so $c_{r,S_t}(G)=0$. Furthermore, if there is a vertex $v$ of degree exactly $r(t-1)$, then for an edge $e$ incident to $v,$ the graph $G'=G-e$ has at least as many colorings as $G$. Indeed, every coloring of $G$ induces a coloring of $G'$ in an injective way, since the color of the other $(r-1)(t-1)-1$ edges incident to $v$ define the color of the edge $e$ uniquely.

\medskip

 On the other hand, if $G$ has two vertices $u$, $v$ of degree less than $\left\lceil \frac{r}{2} \right\rceil\cdot(t-1)$ not joined by an edge, the graph $G'=G+uv$ has at least as many good colorings as $G$, since in every partial coloring of $G'$ that comes from a coloring of $G$, there is at least one free color for the edge $uv$. Therefore, we may assume that all such vertices induce a clique, which implies that there is at most $\left\lceil \frac{r}{2} \right\rceil\cdot(t-1)+1 \leq rt+1$ of them.

\end{proof}

\section{Applying an entropy lemma}

In this section, we will outline the general framework on which our proofs will rely. We start by stating a crucial lemma from \cite{chung1986some}. Before stating it, let us recall the definition of a \emph{projection}. For a set $\mathcal{F}$ of vectors in $F_1 \times \dots \times F_m$ and a subset $S$ of $\{1,\dots,m\}$, the projection of $\mathcal{F}$ on $S$ is defined as $\pi_S(\mathcal{F})$, where $\pi_S:F_1 \times \dots \times F_m \rightarrow\bigotimes_{i\in S} F_i$ is the function that, for every $i \in S$, $(\pi_S(v))_i=v_i$ ($v_i$ denotes the coordinate of the vector $v$ corresponding to the factor $F_i$). To put it simply, the projection of a vector on $S$ ``erases'' its coordinates whose indices do not belong to $S$ and leave the other coordinates unchanged.

\begin{lemma}\label{shearer}
Let $\mathcal{F}$ be a family of vectors in $F_1 \times \dots \times F_m$. Let $\mathcal{G} = \{\mathcal{G}_1, \dots , \mathcal{G}_n\}$ be a collection of subsets of $M = \{1, \dots, m\}$, and suppose that each element $i \in M$ belongs to at least $k$ members of $\mathcal{G}$. For $j = 1, \dots , n$ let $\mathcal{F}_j$ be the set of all projections of the members of $\mathcal{F}$ on $\mathcal{G}_j$. Then

\begin{equation}\label{shearer}
    |\mathcal{F}|^k \leq \prod_{j=1}^n |\mathcal{F}_j|.
\end{equation}

\end{lemma}

In our proofs, we will take $\mathcal{F}$ to be the set of $r$-edge-colorings of a graph $G$ without monochromatic copies of $S_t$. It is a family of vectors in $[r]^{|E(G)|}$, where an edge-coloring $c:E(G)\rightarrow [r]$ is identified with the vector indexed by the edges of $G$ whose value in entry $e\in E(G)$ is $c(e)$.

\medskip

For each $ab$-edge $e_i$ of $G$, we will take a set $\mathcal{G}_i$ to be the set of indices of $e_i$ and the edges incident to it, and we take $2r(t-1)-2-(a+b)$ identical unit sets $\mathcal{G}^1_i,\dots,\mathcal{G}^{2r(t-1)-2-(a+b)}_i$ containing the index of $e_i$. This choice guarantees that each edge is counted $2r(t-1)-3$ times among the sets in $\mathcal{G}$, so we may apply inequality (\ref{shearer}) with $k=2r(t-1)-3$.

\medskip

Let us now estimate the size of the $\mathcal{F}_j$. It is the number of restrictions of $r$-edge-colorings of $G$ without monochromatic $S_t$ to the subgraph spanned by the edges in the set $\mathcal{G}_j$. The number of $r$-edge-colorings without monochromatic $S_t$ of this subgraph is an upper bound for $|\mathcal{F}_j|$. 

\medskip

For the unit sets $\mathcal{G}^i_j$, it is clear that $|\mathcal{F}^i_j|\leq r$. Otherwise, let us denote by $f(x)$ the number of $r$-edge-colorings without monochromatic $S_t$ of a star on $x$ edges in which the color of exactly one edge is fixed (although $f$ depends on $r$ and $t$ as well, we omit these variables from the notation of $f$ as they will be fixed and clear from the context). If we color an $ab$-edge $e_i$ and then the stars hanging on its endpoints, we get $|\mathcal{F}_i|\leq rf(a)f(b)$.

\medskip

Taking into account both types of sets, an $ab$-edge contributes to the right-hand side of (\ref{shearer}) with a factor of $g(a,b)=r^{2r(t-1)-1-(a+b)}f(a)f(b)$.

\medskip

Plugging this bound in $(\ref{shearer})$, we get an optimization problem in terms of the number of $ab$-edges of $G$. This problem would be significantly simplified if we could assume that almost all edges of $G$ are $aa$-edges.

\medskip

This is indeed the case, since whenever we have a pair of independent $ab$-edges ($a \neq b$) $e=uv$ and $f=xy$, say, $d(u)=d(x)=a$ and $d(v)=d(y)=b$, such that $ux$ and $vy$ are not edges (note that there is always a pair of such $ab$ edges if we have more than $a+b$ of them), we may consider the graph $G'$ formed by $G$ by deleting $uv$ and $xy$ and adding $ux$ and $vy$. Note that $G'$ has two less $ab$-edges, one more $aa$-edge and one more $bb$-edge than $G$. On the other hand, the upper bounds on the number of colorings of $G$ and $G'$ given by (\ref{shearer}) are the same, since $g(a,b)^2=g(a,a)\cdot g(b,b)$, and the degrees of the endpoints of all other edges remain unchanged. Therefore, repeating this procedure as long as we can, we may assume that $G$ has at most a constant number of $ab$-edges with $a \neq b$ (bounded, for instance, by $\sum a+b$ over the range $\left\lceil\frac{r}{2} \right\rceil\cdot(t-1) \leq a \neq b \leq r(t-1)+1$). In particular, we may rewrite (\ref{shearer}) as

\begin{align}
    |\mathcal{F}|^{2r(t-1)-3} &\leq c\cdot \prod_{a=\lceil\frac{r}{2}\rceil\cdot(t-1)}^{r(t-1)-1} (r^{2r(t-1)-1-2a}f(a)^2)^{m_a} \nonumber \\
            &= c' \cdot \prod_{a=\lceil\frac{r}{2}\rceil\cdot(t-1)}^{r(t-1)-1}(r^{2r(t-1)-1-2a}f(a)^2)^{av_a/2}\label{bound},
\end{align}

where the range of $a$ in the product comes from Lemma \ref{degrees}.
\medskip

By taking logarithms, it is clear that we are maximizing a linear function of the $v_i$. This means, as $\sum v_i = n$ is constant, that the maximum is attained when all but one of the $v_i$ are zero, and the exceptional $v_i$ corresponds to the value that maximizes the function $g(a)=(r^{2r(t-1)-1-2a}f(a)^2)^a$.

\section{Forbidding small stars in 2-edge-colorings}

In this section, we prove Theorems \ref{S3} and \ref{S4}. Following the setup in the previous section, the proofs are quite straightforward:

\medskip

\begin{proof}[Proof of Theorem \ref{S3}]
By (\ref{bound}), we have the following bound:

\begin{align}
    |\mathcal{F}|^5&\leq c\cdot \prod_{a=2}^3 (2^{7-2a}f(a)^2)^{av_a/2}\\
                    &= c'\cdot 32^{v_2}\cdot 18^{3v_3/2},
\end{align}

since $f(2)=2$ and $f(3)=3$ in this case. The fact that $32<18^{3/2}\approx 76$ concludes the proof.

\end{proof}

\medskip

\begin{proof}[Proof of Theorem \ref{S4}]

In this case, simple computations show that $f(3)=4$, $f(4)=7$ and $f(5)=10$. Therefore, the bound (\ref{bound}) reads as

\begin{equation*}\label{boundS_4number}
    |\mathcal{F}|^9 \leq c\cdot 512^{m_3}\cdot392^{m_4}\cdot200^{m_5}=c'\cdot 512^{3v_3/2}\cdot392^{4v_4/2}\cdot 200^{5v_5/2}.
\end{equation*}

As $512^{3/2}\approx 11585$, $392^{4/2}=153664$ and $200^{5/2}\approx565685$, the maximum is achieved when $v_3=v_4=0$ and $v_5=n$, and the proof is complete.

\end{proof}

\section{Forbidding large monochromatic stars in two-edge-colorings}

In this section, we prove Theorem \ref{b2t}.

\begin{proof}[Proof of Theorem \ref{b2t}]

In this case, $f(x)=\sum_{k=x-t}^{t-2} \binom{x-1}{k}$, since given a star on $x$ edges with one edge colored with color $c$, we may choose at least $x-t$ and at most $t-2$ of the remaining $x-1$ edges to assign $c$ without having a monochromatic $S_t$ in any of the colors.

\medskip

We are done, then, if we find the maximum of $g(a)=\left(2^{4t-5-2a}\left(\sum_{k=a-t}^{t-2} \binom{a-1}{k}\right)^2\right)^{a}$, for $t-1\leq a \leq 2t-3$. We claim that, for $t$ large enough, the maximum value of $g$ is attained for $a=2t-3$.

\medskip

To prove this claim, we will use the following well-known bounds for large $a$ and $t$:

\begin{equation}\label{est1}
    \binom{2t-3}{t-2}\geq 0.9 \cdot \frac{2^{2t-3}}{\sqrt{\pi{t}}}
\end{equation}

and

\begin{equation}\label{est2}
    \binom{a-1}{\lceil{\frac{a-1}{2}}\rceil}\leq 1.01\cdot\frac{2^{a-1}}{\sqrt{\pi{a}}},
\end{equation}

that are consequences of the well-known Stirling formula: $n! \sim \sqrt{2\pi n}(\frac{n}{e})^n.$

\medskip

The first one implies that

\begin{align*}
    g(2t-3) & =\left(2\binom{2t-3}{t-2}^2\right)^{2t-3}\\
            & > \left(\frac{0.9^2\cdot2^{4t-5}}{\pi{t}}\right)^{2t-3}\\
            & >2^{8t^2-2t\log_2{t}-25.92t+O(\log(t))}.
\end{align*}

\medskip

Also, we have $f(a)\leq 2^{a-1}$, since $f(a)$ is a sum of binomial coefficients in the $(a-1)$-st row of Pascal's triangle. Hence,

\begin{equation*}
    g(a)\leq(2^{4t-5-2a}(2^{a-1})^2)^{a}=2^{(4t-7)a}.
\end{equation*}

Suppose first that $a\leq 2t-\log_2{t}$. Then the last inequality implies that

\begin{equation*}
    g(a)\leq2^{(4t-7)(2t-\log_2{t})}=2^{8t^2-4t\log_2{t}+O(t)}\leq g(2t-3)
\end{equation*}

for large $t$.

\medskip

On the other hand, if $2t-\log_2{t}\leq a \leq 2t-4$, notice that, as the central binomial coefficient is the maximum in its row, we have

\begin{equation*}
    f(a)=\sum_{k=a-t}^{t-2} \binom{a-1}{k}\leq (2t-a-1)\binom{a-1}{\lceil{\frac{a-1}{2}}\rceil}\leq 1.01(2t-a-1)\frac{2^{a-1}}{\sqrt{\pi{a}}}, 
\end{equation*}

by (\ref{est2}).

\medskip

The latter estimate implies that

\begin{align*}
    g(a) &\leq (2^{4t-5-2a}(1.01(2t-a-1)\cdot2^{a-1}/\sqrt{\pi{a}})^2)^a\\ 
           &=2^{a(4t-7+2\log_2(2t-a-1)+\log_2(1.01^2/\pi)-\log_2{a})}.
\end{align*}

By taking the derivative (for fixed $t$, with respect to $a$) of the function in the exponent, it is easy to see that this bound on $g$ is increasing for $2t-\log_2{t}\leq a \leq 2t-4$ and large $t$. Therefore, the maximum of the bound in this range is attained for $a=2t-4$, which gives, for large $t$,

\begin{align*}
    g(a)&\leq 2^{(2t-4)(4t-7+2\log_2(3)+\log_2(1.01^2/\pi)-\log_2{(2t-4)})}\\
        &<2^{8t^2-2t\log_2{t}-26t+O(\log(t))}\\
        &<g(2t-3).
\end{align*}
   
Now the fact that $g(2t-3) = \left(2\binom{2t-3}{t-2}^2\right)^{2t-3}$, together with (\ref{bound}), gives the result.

\end{proof}

\section{More colors}

Finally, we prove Theorem \ref{br3}.

\begin{proof}[Proof of Theorem \ref{br3}]

The bound in (\ref{bound}) can be written as

\begin{equation}\label{boundr3'}
    |\mathcal{F}|^{4r-3} \leq c' \prod_{a=r}^{2r-1}(r^{4r-2a-1}f(a)^2)^{av_a/2}.
\end{equation}

Again, all it is left to do is to prove that the maximum of $g(a)=(r^{4r-2a-1}f(a)^2)^a$ is obtained for $a=2r-1$.  With this result, our theorem follows by plugging $v_i=0$ for $i<2r-1$ and $v_{2r-1}=n$ in (\ref{boundr3'}) and by the fact that $f(2r-1)=\frac{(2r-1)!}{2^{r-1}}$.

\medskip

We have, from Stirling's formula,

$$g(2r-1)=\left(\frac{r(2r-1)!^2}{2^{2r-2}}\right)^{2r-1}=r^{8r^2-4(2-\log(2))\frac{r^2}{\log(r)}+o(\frac{r^2}{\log(r)})}.$$

\medskip

We are going to bound $f(a)$ in two different ways and use each of the bounds for a different range of the value of $a$. 

\medskip

First, notice that $f(a)\leq r^{a-1}$, since this is the total number of $r$-colorings of a star with $a-1$ edges. This bound is enough if $a\leq 2r-2r/\log(r)$. Indeed, in this case,

\begin{align*}
    g(a)&\leq(r^{4r-2a-1}\cdot r^{2a-2})^a\\        
    &<r^{(4r-3)(2r-2\frac{r}{\log(r)})}\\
    &=r^{8r^2-8\frac{r^2}{\log(r)}+O(r)}\\
    &<g(2r-1),
\end{align*}

for large $r$.

\medskip

Suppose now that that $a \geq 2r-2r/\log(r)$. Let us divide the colorings counted by $f(a)$ according to the number of times each color appears on it. There are exactly $\frac{(a-1)!}{\prod_{i=1}^r c_i!}$ colorings where the color $i$ appears exactly $c_i$ times, where $0 \leq c_1\leq 1$; $0 \leq c_i\leq 2$, for $i\geq 2$; $\sum_{i=1}^r c_i=a-1$. The number of solutions of this equation can be split according to the value of $c_1$. If $c_1 = 0$, the equation is equivalent to$\sum_{i=1}^r c_i=a-1$, with $0 \leq c_i\leq 2$. If $c_1 = 1$, it is equivalent to $\sum_{i=1}^r c_i=a-2$, with $0 \leq c_i\leq 2$. 

\medskip

Let us consider the first equation. If a solution has exactly $t$ terms equal to 2, then there are exactly $a-1-2t$ terms equal to 1 and $r-a+t$ terms equal to 0. Therefore, there are $\frac{(r-1)!}{t!(a-1-2t)!(r-a+t)!}$ solutions with these many 0, 1 and 2, and those solutions contribute with $\frac{(a-1)!}{\prod_{i=1}^r c_i!}\frac{(r-1)!}{t!(a-1-2t)!(r-a+t)!}=\frac{(a-1)!}{2^t}\frac{(r-1)!}{t!(a-1-2t)!(r-a+t)!}$ to $f(a)$. As the possible values of $t$ range between $a-r$ and $(a-1)/2$, the total contribution of the solutions of the first equation to $f(a)$ is $f_1(a)=\sum_{t=a-r}^{(a-1)/2}\frac{(a-1)!}{2^t}\frac{(r-1)!}{t!(a-1-2t)!(r-a+t)!}$. This sum is bounded from above by $\frac{(a-1)!(r-1)!}{\min_t(t!(a-1-2t)!(r-a+t)!)}\sum_{t = a-r}^{(a-1)/2}\frac{1}{2^t} \leq \frac{(a-1)!(r-1)!}{2^{a-r-1}}\frac{1}{\min_{t}(t!(a-1-2t)!(r-a+t)!)}$, for $a-r \leq t \leq (a-1)/2$.

\medskip

Similarly, the contribution of the second equation is bounded from above by $f_2(a) = \frac{(a-1)!(r-1)!}{2^{a-r-2}}\frac{1}{\min_t( t!(a-2-2t)!(r-a+t+1)!}$, where $a-r-1 \leq t \leq (a-2)/2$.

\medskip

First, let us assume that $2r-a \leq \sqrt{r}$. Considering the ratios of the expressions inside the minimum above for consecutive values of $t$, $\frac{(t+1)!(a-3-2t)!(r-a+t+1)!}{t!(a-1-2t)!(r-a+t)!}$  and $\frac{(t+1)!(a-4-2t)!(r-a+t+2)!}{t!(a-2-2t)!(r-a+t+1)!}$, it is possible to prove that both minimum values are attained on the left endpoints of the corresponding ranges of $t$, namely $t = a-r$ and $t=a-r-1$.

\medskip

Hence, we can rewrite the upper bounds for the contributions of the equations as $f_1(a) \leq \frac{(a-1)!(r-1)!}{2^{a-r-1}}\frac{1}{(a-r)!(2r-a-1)!}$ and $f_2(a) \leq \frac{(a-1)!(r-1)!}{2^{a-r-2}}\frac{1}{(a-r-1)!(2r-a)!}$. The two bounds together imply $f(a) = f_1(a)+f_2(a) \leq 2f_1(a)+f_2(a) \leq \frac{(a-1)!r!}{2^{a-r-2}(a-r)!(2r-a)!}$. Hence we have the following estimate for $g$:

\begin{equation}\label{secondest}
    g(a)\leq \left(\frac{r^{4r-2a-1}(a-1)!^2r!^2}{2^{2a-2r-4}(a-r)!^2(2r-a)!^2}\right)^a.
\end{equation}

\medskip

We will prove that the upper bound for $g(a)$ in (\ref{secondest}), call it $h(a)$, is increasing with $a$ in the range $2r-2r/\log(r) \leq a \leq 2r-5$, and that for $a=2r-4$ it gives a value smaller than $g(2r-1)$. The cases $a=2r-3$ and $a=2r-2$ will be dealt with separately.

\medskip

It is a simple exercise to compute that $f(2r-2) = r(2r-2)!/2^{r-1}$ and $f(2r-3)=(r+1)(2r-2)!/(3\cdot2^{r-1})$. Thus, $g(2r-2) = (r^5(2r-2)!^2/2^{2r-2})^{2r-2}$ and $g(2r-3) = (r^5(r+1)^2(2r-2)!^2/(9\cdot 2^{2r-2}))^{2r-3}$. Hence, applying Stirling's formula, the following estimates hold as $r \to \infty$, where the $c_i>0$ and $c_i'$ are constants:

\begin{align*}
    \frac{g(2r-1)}{g(2r-2)}&=\left(\frac{r(2r-1)!^2}{2^{2r-2}}\right)^{2r-1}\cdot\left(\frac{r^{5}(2r-2)!^2}{2^{2r-2}}\right)^{-(2r-2)}
    \\
    & = \frac{(2r-1)^{4r-2}\cdot(2r-2)!^2}{r^{8r-9}\cdot 2^{2r-2}}\\
        &\sim c_1 \cdot r^{c'_1} \cdot\frac{2^{6r}}{e^{4r}}\\
        &\rightarrow\infty
\end{align*}

and

\begin{align*}
    \frac{g(2r-1)}{g(2r-3)}&=\left(\frac{r(2r-1)!^2}{2^{2r-2}}\right)^{2r-1}\cdot\left(\frac{r^{5}(r+1)^2(2r-2)!^2}{9 \cdot 2^{2r-2}}\right)^{-(2r-3)}
    \\
    & \sim \frac{c_2 \cdot 9^{2r-3} \cdot(2r-1)^{4r-2}\cdot(2r-2)!^4}{r^{12r-20}\cdot 2^{4r-4}}\\
        &\sim c_3 \cdot r^{c'_3} \cdot\frac{9^{2r}\cdot 2^{8r}}{e^{8r}}\\
        &\rightarrow\infty,
\end{align*}

since $2^6 > e^4$ and $9^2 \cdot 2^8 > e^8$.

\medskip

On the other hand, plugging $a=2r-4$ in (\ref{secondest}), we get

\begin{equation*}
    g(2r-4)\leq\left(\frac{r^{15}(2r-5)!^2}{24^2\cdot2^{2r-12}}\right)^{2r-4}.
\end{equation*}

Again, Stirling's formula implies that, for some positive constant $c$ and some constant $c'$,

\begin{align*}
    \frac{g(2r-1)}{g(2r-4)}&\geq\left(\frac{r(2r-1)!^2}{2^{2r-2}}\right)^{2r-1}\cdot\left(\frac{r^{15}(2r-5)!^2}{24^2\cdot2^{2r-12}}\right)^{-(2r-4)}
    \\
    & = \frac{((2r-1)(2r-2)(2r-3)(2r-4))^{4r-2}\cdot 24^{4r-8}(2r-5)!^6}{r^{28r-59}\cdot 2^{26r-46}}\\
        &\sim c \cdot r^{c'} \cdot\frac{24^{4r}\cdot 2^{2r}}{e^{12r}}\\
        &\rightarrow\infty,
\end{align*}

when $r \rightarrow\infty$, since $24^4\cdot2^2 > e^{12}$.

\medskip

To prove that $h$ is increasing in this range, we first calculate $h(a+1)/h(a)$:

\begin{equation*}
    h(a+1)/h(a) = \frac{r^{4r-4a-3}r!^2a^{2a}a!^2(2r-a)^{2a}}{2^{4a-2r-2}(a-r+1)^{2a}(a-r+1)!^2(2r-a-1)!^2}.
\end{equation*}

Computing the logarithm of each term in this expression to the base $r$, we get, using that $\log(n!) = n\log(n)-n+O(\log(n))$, writing $\alpha = 2r-a$, and ignoring $o(r/\log(r))$ terms in the equations that follow, that $\log_r(r^{4r-4a-3})=4\alpha-4r$, $\log_r(r!^2)= 2r - 2r/\log(r)$, $\log_r(a^{2a})= 2a+(2\log(2))a/\log(r) $, $\log_r(a!^2)= 2a+2(\log(2)-1)a/\log(r)$, $\log_r((2r-a)^{2a})= 2a\log(\alpha)/\log(r)$, $\log_r(2^{4a-2r-2})= (4\log(2))a/\log(r)-(2\log(2))r/\log(r)$, $\log_r((a-r+1)^{2a})= 2a$, $\log_r((a-r+1)!^2)= 2(a-r)-2a/\log(r)+2r/\log(r) $, and $\log_r((2r-a-1)!^2)= 2(\alpha-1)\log(\alpha-1)/\log(r)$. Putting all these expression together, we get the following estimate for $\log_r(h(a+1)/h(a))$:

\begin{equation}\label{estimate}
    (2\log(2)-4)\frac{r}{\log(r)}+2a\frac{\log(\alpha)}{\log(r)}-2(\alpha-1)\frac{\log(\alpha-1)}{\log(r)}+4\alpha+o\left(\frac{r}{\log(r)}\right).
\end{equation}

But, recalling that $\alpha \geq 3$ and $a \geq 2r-2r/\log(r)$, we get that this expression is at least

\begin{align*}
    &(2\log(2)-4)\frac{r}{\log(r)}+2a\frac{\log(\alpha)}{\log(r)}-2(\alpha-1)\frac{\log(\alpha-1)}{\log(r)}+4\alpha \geq\\
     &(2\log(2)-4)\frac{r}{\log(r)} + 2(2r-2
     \alpha)\frac{\log(\alpha)}{\log(r)} + 12 \geq\\
     &  (4\log(3)+2\log(2)-4)\frac{r}{\log(r)}+o\left(\frac{r}{\log(r)}\right),
\end{align*}

which is positive for large $r$, since $4\log(3)+2\log(2)-4 > 0$. This proves that $h$ is increasing in this range and completes the proof in the case $a \geq 2r-\sqrt{r}$.

\medskip

Suppose, on the other hand, that $\sqrt{r} < \alpha < {2r/\log(r)}$. In this case, the minimum of $t!(a-1-2t)!(r-a+t)!$, where $a-r \leq t \leq (a-1)/2$ and of $t!(a-2-2t)!(r-a+t+1)!$, where $a-r-1 \leq t \leq (a-2)/2$, is not attained on the left end of the respective intervals, but in the root of a quadratic equation that lies in the middle of the interval, where the ratio of consecutive terms is equal to 1.

\medskip

In the case of the first equation, this corresponds to the smallest root of the equation $\frac{(t+1)!(a-3-2t)!(r-a+t+1)!}{t!(a-1-2t)!(r-a+t)!} = 1$, or $(t+1)(r-a+t+1)=(a-1-2t)(a-2-2t)$, which is

\begin{equation*}
    t_0 = \frac{1}{6}(3 a + r - 4  -\sqrt{-3 a^2 + 6 a r + r^2 + 4 r + 4}) = a-r+k_1,
\end{equation*}

where $k_1 = O\left(\frac{\alpha^2}{r}\right)$.

\medskip

Hence, using the elementary estimates $(\frac{a}{b})^b \leq \binom{a}{b} \leq (\frac{ea}{b})^b$ and $\binom{2k}{k} \leq 4^k$, we can prove that 

\begin{align*}
    \frac{t_0!(a-1-2t_0)!(r-a+t_0)!}{(a-r)!(a-1-2(a-r))!(r-a+(a-r))!} & =\frac{\binom{r-\alpha+k_1}{k_1}}{\binom{2k_1}{k_1}\binom{\alpha-1}{2k_1}}\\
    &\geq \frac{(r-\alpha+k_1)^{k_1}\cdot k_1^{k_1}}{e^{2k_1}\cdot \alpha^{2k_1}},
\end{align*}

so the value of the minimum of $t!(a-1-2t)!(r-a+t)!$ is bounded from below by its value on the left endpoint times the factor in the right-hand side of the inequality above. Call it $m_1$.

\medskip

A similar calculation for the second equation proves that 

\begin{align*}
    & \frac{t_1!(a-2-2t_1)!(r-a+t_1+1)!}{(a-r-1)!(a-2-2(a-r-1))!(r-a+(a-r-1)+1)!}  \geq \\  &\frac{(r-\alpha-1+k_2)^{k_2}\cdot k_2^{k_2}}{e^{2k_2}\cdot \alpha^{2k_2}},
\end{align*}

where $k_2 = O\left(\frac{\alpha^2}{r}\right)$. Call the right-hand side of the inequality above $m_2$.

\medskip

These two estimates together imply that

\begin{equation*}
    f(a) \leq \frac{(a-1)!r!}{2^{a-r-2}(a-r)!(2r-a)!}\cdot M,
\end{equation*}

where $M = \max\{m_1,m_2\}$, and hence

\begin{equation*}
    g(a) \leq \left(\frac{r^{4r-2a-1}(a-1)!^2r!^2}{2^{2a-2r-4}(a-r)!^2(2r-a)!^2}\right)^a\cdot M^{2a} = h(a)\cdot M^{2a}.
\end{equation*}

This is the bound in (\ref{secondest}) with an extra term $M^{2a}$, where $M^{2a} = r^{O(\alpha^2/r)}$.

\medskip

Applying the estimate (\ref{estimate}), we get that

\begin{align*}
    \log_r\left(\frac{h(2r-2)}{h(a)}\right) & = \log_r\left(\prod_{i=3}^\alpha \frac{h(2r-i+1)}{h(2r-i)}\right) \\
    & = \sum_{i=3}^\alpha \log_r\left(\frac{h(2r-i+1)}{h(2r-i)}\right) \\
    & \geq cr\alpha,
\end{align*}

for some $c>0$.

\medskip

Finally, as $\alpha^2/r = o(r\alpha) $, this inequality implies that 

\begin{align*}
    g(a) & = h(a) \cdot M^{2a} \\
        & \leq \frac{h(2r-2) \cdot M^{2a}}{r^{cr\alpha}} \\
        & \leq \frac{g(2r-1) \cdot r^{O(\alpha^2/r)}}{r^{cr\alpha}}\\
        & \leq g(2r-1)
\end{align*}

and concludes the proof.

\end{proof}

\section{Final remarks and open problems}\label{open}

Our argument could be generalized by taking the sets $\mathcal{G}_j$ to include bigger neighborhoods of the edge $e_j$. However, in this case, new technical problems arise when we try to estimate the $|\mathcal{F}_i|$. Somewhat better results could be achieved, but we do not believe that they get substantially closer to the lower bounds.

\medskip

We conjecture that $b_{2,S_3}=\sqrt[6]{102}$, i.e., the union of disjoint $K_{3,3}$'s is the graph with the largest number of 2-edge-colorings without monochromatic $S_3$. In general, for 2-colorings forbidding monochromatic stars of a fixed size, we think, in agreement with \cite{hoppen2014edge}, that the extremal configuration is given by a collection of copies of a fixed (possibly complete bipartite) graph of constant size.

\section{Acknowledgments}

The first author wants to thank Professor Yoshiharu Kohayakawa for introducing him this subject during his visit as a PhD exchange student at University of São Paulo (Brazil). The research of the second author was partially supported by NKFIH grants \#116769, \#117879 and \#126853. The authors thank the reviewers for their valuable comments and suggestions.

\bibliographystyle{acm}
\bibliography{ref}

\begin{thebibliography}{10}

\bibitem{alon2004number}
{\sc Alon, N., Balogh, J., Keevash, P., and Sudakov, B.}
\newblock The number of edge colorings with no monochromatic cliques.
\newblock {\em Journal of the London Mathematical Society 70}, 2 (2004),
  273--288.

\bibitem{benevides2017edge}
{\sc Benevides, F.~S., Hoppen, C., and Sampaio, R.~M.}
\newblock Edge-colorings of graphs avoiding complete graphs with a prescribed
  coloring.
\newblock {\em Discrete Mathematics 340}, 9 (2017), 2143--2160.

\bibitem{chung1986some}
{\sc Chung, F.~R., Graham, R.~L., Frankl, P., and Shearer, J.~B.}
\newblock Some intersection theorems for ordered sets and graphs.
\newblock {\em Journal of Combinatorial Theory, Series A 43}, 1 (1986), 23--37.

\bibitem{erdos1992some}
{\sc Erd{\H o}s, P.}
\newblock Some of my favourite problems in various branches of combinatorics.
\newblock {\em Le Matematiche 47}, 2 (1992), 231--240.

\bibitem{fekete1923verteilung}
{\sc Fekete, M.}
\newblock {\"U}ber die {V}erteilung der {W}urzeln bei gewissen algebraischen
  {G}leichungen mit ganzzahligen {K}oeffizienten.
\newblock {\em Mathematische Zeitschrift 17}, 1 (1923), 228--249.

\bibitem{han2018improved}
{\sc H{\`a}n, H., and Jim{\'e}nez, A.}
\newblock Improved bound on the maximum number of clique-free colorings with
  two and three colors.
\newblock {\em SIAM Journal on Discrete Mathematics 32}, 2 (2018), 1364--1368.

\bibitem{hoppen2012edge}
{\sc Hoppen, C., Kohayakawa, Y., and Lefmann, H.}
\newblock Edge colourings of graphs avoiding monochromatic matchings of a given
  size.
\newblock {\em Combinatorics, Probability and Computing 21}, 1-2 (2012),
  203--218.

\bibitem{hoppen2014edge}
{\sc Hoppen, C., Kohayakawa, Y., and Lefmann, H.}
\newblock Edge-colorings of graphs avoiding fixed monochromatic subgraphs with
  linear {T}ur{\'a}n number.
\newblock {\em European Journal of Combinatorics 35\/} (2014), 354--373.

\bibitem{pikhurko2017erdHos}
{\sc Pikhurko, O., Staden, K., and Yilma, Z.~B.}
\newblock The {E}rd{\H{o}}s--{R}othschild problem on edge-colourings with
  forbidden monochromatic cliques.
\newblock In {\em Mathematical proceedings of the cambridge philosophical
  society\/} (2017), vol.~163, Cambridge University Press, pp.~341--356.

\bibitem{pikhurko2012maximum}
{\sc Pikhurko, O., and Yilma, Z.~B.}
\newblock The maximum number of ${K}_3$-free and ${K}_4$-free edge 4-colorings.
\newblock {\em Journal of the London Mathematical Society 85}, 3 (2012),
  593--615.

\bibitem{yuster1996number}
{\sc Yuster, R.}
\newblock The number of edge colorings with no monochromatic triangle.
\newblock {\em Journal of Graph Theory 21}, 4 (1996), 441--452.

\end{thebibliography}

\end{document}